\documentclass[11pt]{amsart}
\usepackage{amsmath,amssymb,amsthm}
\usepackage{enumitem}
\usepackage{needspace}
\usepackage{tikz}
\usepackage{xcolor}
\tikzset{
  quiver/.style={>=stealth,semithick},
  quiver vertex/.style={circle,draw,fill=white,inner sep=1.5pt,
    minimum size=15pt,font=\small},
  quiver label/.style={font=\scriptsize,inner sep=1.5pt}
}
\usepackage[colorlinks=true,citecolor=red,linkcolor=blue,urlcolor=blue]{hyperref}
\newcommand{\doi}[1]{\href{https://doi.org/#1}{doi:\nolinkurl{#1}}}
\setlist[enumerate]{itemsep=1pt,topsep=3pt}
\numberwithin{equation}{section}
\newtheorem{theorem}{Theorem}[section]
\newtheorem{proposition}[theorem]{Proposition}
\newtheorem{lemma}[theorem]{Lemma}
\newtheorem{corollary}[theorem]{Corollary}
\newtheorem{theoremA}{Theorem}

\theoremstyle{definition}
\newtheorem{definition}[theorem]{Definition}
\newtheorem{example}[theorem]{Example}
\newtheorem{remark}[theorem]{Remark}

\DeclareMathOperator{\Ext}{Ext}
\DeclareMathOperator{\rank}{rank}
\DeclareMathOperator{\gldim}{gldim}
\newcommand{\kk}{\mathbf{k}}
\newcommand{\Z}{\mathbb Z}
\newcommand{\U}{\widetilde A}

\begin{document}
\title[Derived tame monomial algebras]{Derived tame monomial algebras}

\author{Marlon Pimenta Fonseca}
\address{Department of Mathematics and Applied Mathematics,
Federal University of Lavras (UFLA) -- Brazil}
\email{marlon.fonseca@ufla.br}

\author{Zheng Xin}
\address{School of Mathematics and Statistics, Shaanxi Normal University, Xi'an 710062, China}
\email{xinzheng1314@yeah.net}

\date{}

\begin{abstract}
We prove that a finite-dimensional connected monomial algebra over
an algebraically closed field is derived tame if and only if every
finite connected convex subcategory of its universal cover has
non-negative Euler form. Derived tameness is also equivalent to
tameness of its repetitive category and trivial extension.
For one-cycle algebras, this condition is equivalent to the non-negativity
on $\mathbb R$ of $|Q_0|$ explicitly computable polynomials in
$\mathbb Q[t]$. This gives a finite algorithm for deciding derived tameness.
We also construct an infinite family of derived-tame one-cycle string
algebras with minimal cubic relations, no two of which are derived equivalent.
\end{abstract}

\maketitle
\setcounter{tocdepth}{2} 

\tableofcontents
\section*{Introduction}

Br\"ustle and Geiss independently proved that a finite tree
algebra is derived tame if and only if its Euler form is
non-negative \cite{Bru01,Gei02}. For triangular monomial algebras, Novoa and de la Peña proved that derived tameness is equivalent to the tameness of the trivial extension~\cite[Theorem~2.1]{NdP01}, and that these conditions hold precisely when every finite convex subcategory of the universal cover has a non-negative Euler form~\cite[Corollary~3.2]{NdP01}. Derived-tame quadratic string algebras and Nakayama algebras
were classified in \cite{Fon22} and \cite{BGV21}, respectively.
Bekkert and Merklen described the indecomposable objects in
the derived categories of gentle algebras \cite{BM03}.

Theorem~\ref{thm:monomial} extends the local Euler criterion and the equivalences in~\cite[Theorem~2.1 and Corollary~3.2]{NdP01} to all finite-dimensional connected monomial algebras, accommodating arbitrary oriented cycles and algebras with infinite global dimension. When the global dimension is infinite, the usual Ext formula
need not define an Euler form on $K_0(A)$. To overcome this limitation, our proof operates locally on finite connected convex full subcategories $B \subseteq \tilde{A}$ of the universal cover. As finite tree subcategories, each $B$ has finite global dimension, ensuring that the local Euler form $\chi_B$ is well-defined. Under the local Euler condition ($\chi_B \ge 0$), the repetitive categories $\hat{B}$ are tame by the theorem of Geiss and Krause~\cite{GK02}. Lemma~\ref{lem:repetitive-local} then extends this tameness to the repetitive category $\widehat{\tilde{A}}$ of the cover, and the covering theorem establishes tameness of $\widehat{A}$, yielding derived tameness of $A$.


The local Euler criterion involves infinitely many finite convex
subcategories. For a one-cycle algebra, Theorem~\ref{thm:fourier}
replaces these checks by a finite computation from the quiver and
its monomial relations. The deck group is $\Z$, so nonzero paths
between vertex lifts define a Hermitian Laurent polynomial matrix.
The relation between the Euler and Cartan forms, together with a
Fourier calculation, reduces the local Euler condition to positive
semidefiniteness of this matrix on the unit circle. A rational
parametrization of the circle then gives $|Q_0|$ polynomials in
$\mathbb Q[t]$ whose non-negativity on $\mathbb R$ decides derived
tameness. Their degrees are bounded in terms of the longest nonzero path.
The criterion applies to one-cycle
monomial algebras with relations of arbitrary length, including
non-quadratic string algebras.
When the unique cycle is not oriented,
Corollary~\ref{cor:nonoriented-cycle} gives test polynomials
of degree at most $2j$ for $1\leq j\leq|Q_0|$.

The two main results are as follows. Let $\kk$ be an algebraically
closed field and let $A=\kk Q/I$ be finite-dimensional, basic and
connected. Write $\U$ for its universal bound-quiver cover,
$\widehat A$ for its repetitive category, and
$\mathrm T(A)=A\ltimes DA$ for its trivial extension.
We call $A$ \emph{one-cycle} if $|Q_1|-|Q_0|+1=1$.
The underlying undirected graph retains any loops and multiple
edges of $Q$; in the one-cycle case, its cycle may have any orientation.
For a finite tree algebra $B$, its Euler form is non-negative
if $\chi_B(x)\geq0$ for every $x\in K_0(B)$.

\begin{theoremA}\label{thm:monomial}
For a finite-dimensional connected monomial algebra $A$, the
following conditions are equivalent:
\begin{enumerate}[label=\textnormal{(\roman*)}]
\item $A$ is derived tame;
\item every finite connected convex full subcategory $B\subseteq\U$
has non-negative Euler form;
\item $\widehat{\tilde{A}}$ is tame;
\item $\widehat A$ is tame;
\item $\mathrm T(A)$ is tame.
\end{enumerate}
\end{theoremA}

When the underlying graph has a cycle, these conditions are also
equivalent to the condition that every finite connected convex full
subcategory of $\U$ is derived equivalent to a semichain algebra
(Proposition~\ref{prop:semichain}).

\begin{theoremA}\label{thm:main}
Let $A=\kk Q/I$ be a finite-dimensional connected one-cycle
monomial algebra, and form $H(z)$ from the path counts in
\eqref{eq:fourier-symbol} using spanning-tree lifts.
Let $R_Q$ be the arrow ideal, and set
$q=|Q_0|$, $L=\min\{\ell\geq2:R_Q^\ell\subseteq I\}$ and $M=L-1$.
Now, define
\[
K(t)=(1+t^2)^{M}H\!\left(\frac{1+\mathrm i t}{1-\mathrm i t}\right),
\qquad
\det(sI_q+K(t))=s^q+\sum_{j=1}^q c_j(t)s^{q-j}.
\]
The following conditions are equivalent:
\begin{enumerate}[label=\textnormal{(\roman*)}]
\item $A$ is derived tame;
\item $H(z)$ is positive semidefinite for every $|z|=1$;
\item $c_j(t)\geq0$ for every $t\in\mathbb R$ and $1\leq j\leq q$.
\end{enumerate}
Each $c_j$ lies in $\mathbb Q[t]$ and has degree at most
$2j(L-1)$ when nonzero. These polynomials are computable from
$Q$ and a finite list of monomial generators of $I$, and their
non-negativity gives a finite test for derived tameness.
\end{theoremA}

\section*{Acknowledgments}
Zheng Xin is grateful to his supervisor, Prof. Wen Chang, for his academic guidance and support. Zheng Xin also thanks Marlon Pimenta Fonseca for patiently and thoroughly answering all his questions during his study of \cite{Fon22}.

\section{Derived tameness and tree algebras}\label{sec:prelim}

Throughout this paper, $\kk$ denotes an algebraically closed field and $D = \operatorname{Hom}_k(-, k)$ is the standard duality. Unless otherwise stated, all algebras are assumed to be finite-dimensional, basic, and connected. We work with finitely generated right modules; for an algebra $A$, we denote by $\operatorname{mod} A$ the category of finitely generated right $A$-modules, by $D^b(A)=D^b(\operatorname{mod} A)$ its bounded derived category, and write $K^{-,b}(\operatorname{proj} A)$ for the homotopy category of right-bounded complexes of finitely generated projective $A$-modules with bounded cohomology. Composition of paths in quivers is written from left to right.  The sets of vertices and arrows of
$Q$ are denoted by $Q_0$ and $Q_1$, respectively. Its underlying
undirected graph retains loops and multiple edges.

Let $Q$ be a finite quiver, let $R_Q \subseteq \kk Q$ be the arrow ideal, and let $I \subseteq \kk Q$ be an admissible ideal ($R_Q^L \subseteq I \subseteq R_Q^2$ for some $L \ge 2$) generated by paths. The algebra $A = \kk Q/I$ is called a \emph{monomial algebra}.

The classes of paths outside $I$, including the trivial paths $e_v$,
form a basis of $A$. We call these classes the \emph{nonzero paths} of
$A$. Every nonzero path has length less than $L$. This basis will be
used in the covering arguments and in the matrices of
Section~\ref{sec:fourier}.

\begin{definition}\label{def:string-gentle}
Let $A=\kk Q/I$ be a monomial algebra.
\begin{enumerate}
    \item We call $A$ a \emph{string algebra} if at most two arrows
    start and at most two arrows end at each vertex, and every arrow
    $\alpha$ has at most one successor $\beta$ with
    $\alpha\beta\notin I$ and at most one predecessor $\gamma$ with
    $\gamma\alpha\notin I$. Only composable pairs of arrows are
    considered; see~\cite{BR87}.

    \item We call $A$ \emph{gentle} if it is a string algebra, $I$ is
    generated by paths of length two, and every arrow has at most one
    successor and at most one predecessor whose composition belongs to
    $I$.

    \item We call $A$ \emph{triangular} if $Q$ has no oriented cycles.

    \item We call $A$ a \emph{tree algebra} if the underlying undirected graph of $Q$ is a tree.
\end{enumerate}
\end{definition}

\begin{definition}\label{def:derived-representation-type}
\cite[Definition~1.2]{BD03} Let $A = \kk Q/I$ be a finite-dimensional algebra.
\begin{enumerate}
    \item The \emph{vector rank} of a bounded complex of projective $A$-modules $P_\bullet$ is the sequence $\mathbf{r}_\bullet(P_\bullet) = (\mathbf{r}(P_n))_{n\in\mathbb{Z}}$, where $\mathbf{r}(P_n) = (r_v)_{v\in Q_0}$ records the multiplicity of each indecomposable projective $e_v A$ in $P_n$.

    \item A \emph{rational family} of minimal complexes over $A$ is a bounded complex $(P_\bullet, d_\bullet)$ of projective $(A \otimes_\kk \Lambda)$-modules, where $\Lambda = \kk[t, f(t)^{-1}]$ ($f \neq 0$), such that $\operatorname{Im} d_n \subseteq (\operatorname{rad} A) P_{n-1}$. Its \emph{specialization} at $(m, \lambda)$ (for $m \ge 1$ and $f(\lambda) \neq 0$) is the complex of $A$-modules
    \[
    P_\bullet(m, \lambda) = P_\bullet \otimes_\Lambda \frac{k[t]}{(t-\lambda)^m}.
    \]

    \item The algebra $A$ is \emph{derived tame} if, for each vector rank $\mathbf{r}_\bullet$, all indecomposable minimal complexes of rank $\mathbf{r}_\bullet$, except for finitely many isomorphism classes, are isomorphic to $P_\bullet(m, \lambda)$ for some $m, \lambda$ and some family $P_\bullet$ belonging to a finite set $\mathfrak{P}(\mathbf{r}_\bullet)$ of rational families.

    \item The algebra $A$ is \emph{derived wild} if there exists a bounded complex $P_\bullet$ of projective $(A \otimes_\kk \Sigma)$-modules, where $\Sigma = \kk\langle x, y\rangle$, such that the functor 
    \[
    - \otimes_\Sigma P_\bullet \colon \operatorname{mod} \Sigma \longrightarrow K^{-,b}(\operatorname{proj} A)
    \]
    preserves indecomposability and reflects isomorphism classes.
\end{enumerate}
\end{definition}

The definition above is
equivalent to the formulation in terms of cohomology dimensions in
\cite{GK02}, see the discussion preceding
\cite[Definition~1.2]{BD03}. We also use the tame--wild dichotomy
\cite[Theorem~1.3]{BD03} and invariance of derived tameness under
derived equivalence \cite[Theorem~5.1]{GK02}.

\begin{definition}\label{def:locally-bounded}
A $\kk$-linear category $R$ is called \emph{locally bounded} if
\begin{enumerate}
    \item distinct objects of $R$ are non-isomorphic;
    \item every endomorphism ring of an object is local;
    \item every morphism space is finite-dimensional; and
    \item each object has nonzero morphisms to or from only finitely
    many objects.
\end{enumerate}
Write $R_0$ for the set of objects. A finite-dimensional $R$-module
has a dimension vector of finite support.
\end{definition}

\begin{definition}\label{def:tame-category}
Let $R$ be a locally bounded $\kk$-linear category. A \emph{one-parameter family} of $R$-modules is a representation $\mathcal{M}$ over a ring $\Lambda = \kk[t, f(t)^{-1}]$, where $f \neq 0$, such that $\mathcal{M}$ has finite support, every $\mathcal{M}(x)$ is a finite free $\Lambda$-module, and all structure maps are $\Lambda$-linear. For $\lambda \in k$ with $f(\lambda) \ne 0$, its \emph{fibre} at $\lambda$ is
\[
\mathcal{M}_\lambda = \mathcal{M} \otimes_\Lambda \Lambda/(t-\lambda).
\]
The dimension vector of the family is the vector of the ranks of its components. We call $R$ \emph{tame} if, for each fixed dimension vector, all but finitely many indecomposable $R$-modules are fibres of finitely many one-parameter families. Representation-finite categories are included in this definition.
\end{definition}

\begin{definition}\label{def:convex}
A full subquiver $Q'$ of $Q$ is called \emph{convex} if every oriented
path in $Q$ whose endpoints lie in $Q'_0$ is entirely contained in
$Q'$, including paths that belong to $I$ (equivalently, paths that are
zero modulo $I$). The associated convex full subcategory has the
induced quiver and the induced relations.
\end{definition}

\begin{lemma}\label{lem:cover}
For a finite-dimensional connected monomial algebra $A=\kk Q/I$,
the universal covering quiver
is a locally finite tree, and its deck group is free of rank
$|Q_1|-|Q_0|+1$. The covering category $\U$ is locally bounded.
Every finite set of its vertices lies in a finite connected convex
full subcategory of finite global dimension.
\end{lemma}
\begin{proof}
A monomial relation does not identify distinct paths in the
homotopy relation defining $\pi_1(Q,I)$. Thus $\pi_1(Q,I)$ is the
fundamental group of the underlying graph. The universal graph
cover, with the lifted relations, gives the bound-quiver cover
$\U$ \cite[Remark~2.5]{Fon22}. Its deck group $G$ acts freely on
vertices, and $\U/G\cong A$.

Each vertex of the covering tree is incident with finitely many
arrows, and every nonzero path has length less than $L$.
Only finitely many nonzero paths can therefore start or end at a
fixed vertex. Thus $\U$ is locally bounded.

Given a nonempty finite set of vertices, take the union of the
undirected paths joining them. The result is a finite connected
subtree. Every oriented path between its vertices lies in this
subtree, so the corresponding full subcategory $B$ is convex.
Order the vertices so that every arrow goes from a smaller index
to a larger one. The composition factors of
$\operatorname{rad}(e_iB)$ then lie at vertices with larger
indices. Induction, starting at the last vertex, shows that each
simple $B$-module has finite projective dimension.
Hence $B$ has finite global dimension.
\end{proof}

For any finite-dimensional basic algebra $A$, put
$K_0(A)=K_0(\operatorname{mod}A)$. This group is free on the
classes of the simple modules, so $\rank K_0(A)=|Q_0|$.
The class of a bounded complex is the alternating sum of the
classes of its cohomology modules. This identifies $K_0(A)$ with
$K_0(D^b(A))$, so derived equivalence preserves
its rank even when the global dimension is infinite.

For a finite tree algebra $B$, the global dimension is finite.
Its Euler pairing and quadratic form are therefore given by
\[
\langle[X],[Y]\rangle_B
=\sum_{j\geq0}(-1)^j\dim_\kk\Ext_B^j(X,Y),
\qquad \chi_B(x)=\langle x,x\rangle_B.
\]
The sum is finite and defines a bilinear form on $K_0(B)$.
We write $\chi_B\geq0$ if it is non-negative on every element
of $K_0(B)$, including vectors with negative coordinates in the
basis of simple classes. This is equivalent to non-negativity on
$K_0(B)\otimes_\Z\mathbb R$, since homogeneity gives non-negativity
on rational vectors and continuity extends it to real vectors.

For $n\geq1$ and $0\leq m\leq n$, the \emph{semichain algebra}
$S(n,m)$ is the incidence algebra of a poset with $n$ levels.
The first $m$ levels contain two incomparable elements each; the
remaining levels are singletons. Every element of an earlier
level is smaller than every element of a later level. Recall
that the incidence algebra has basis $e_{xy}$ for $x\leq y$, with
$e_{xy}e_{yz}=e_{xz}$ and all other basis products zero.
We use the connected members of this family
\cite[p.~302]{Bru01}, excluding $S(1,1)\cong\kk\times\kk$.

\begin{theorem}\label{thm:tree}
A finite connected tree algebra $B$ is derived tame if and only
if $\chi_B\geq0$. In that case $B$ is derived equivalent to a
semichain algebra, a hereditary algebra of type $E_p$ or
$\widetilde E_p$ with $p=6,7,8$, or a tubular algebra. Semichain
algebras are derived tame.

If $\chi_B\geq0$ and $B$ contains a convex subcategory of one of
these exceptional or tubular derived types, then $B$ itself has
one of these types; in particular, $\rank K_0(B)\leq10$.
\end{theorem}

The classification is due to Br\"ustle
\cite[Theorems~1.1 and~1.2]{Bru01}. The assertion about convex
subcategories follows from \cite[Proposition~2.1]{BBD00}.
The hereditary types have at most nine vertices, and the tubular
types have Grothendieck group ranks $6,8,9,10$. These ranks give
the bound used in Proposition~\ref{prop:semichain}.

\section{A covering criterion for derived tameness}
\label{sec:covering}

For a locally bounded category $R$, the \emph{repetitive category}
$\widehat R$ has objects $(x,i)\in R_0\times\Z$ and morphism spaces
\[
\widehat R((x,i),(y,j))=
\begin{cases}
R(x,y),&j=i,\\
DR(y,x),&j=i+1,\\
0,&\text{otherwise}.
\end{cases}
\]
For a bound-quiver category, $R(x,y)$ is the space of paths from
$x$ to $y$ modulo the relations. Each layer is a copy of $R$.
Composition within a layer is that of $R$, and composition between
adjacent layers uses the natural bimodule structure on $DR$.
The product of two morphisms in the dual spaces is zero.

The \emph{trivial extension} of a finite-dimensional algebra $A$ is
$\mathrm T(A)=A\ltimes DA$. Its underlying vector space is $A\oplus DA$,
with multiplication
\[
(a,f)(b,g)=(ab,ag+fb).
\]
The layer translation $\nu(x,i)=(x,i+1)$ gives
$\widehat A/\langle\nu\rangle\cong\mathrm T(A)$.

The next lemma uses finite quotient categories to study modules
with fixed support in a repetitive category.

\begin{lemma}\label{lem:repetitive-local}
Let $R$ be locally bounded. Suppose that every finite subset of
$R_0$ is contained in a finite full subcategory $B$ for which
$\widehat B$ is tame. Then $\widehat R$ is locally bounded and tame.
\end{lemma}

\begin{proof}
The displayed morphism spaces show that $\widehat R$ is locally
bounded. Fix a finite set $S\subseteq\widehat R_0$ and choose
$B$ containing the projection of $S$ to $R_0$.
Since $B$ is full, $\widehat B$ is the full subcategory of
$\widehat R$ on $B_0\times\Z$. Both categories have the same
full subcategory $C$ on $S$.

Let $J_B$ be the ideal of $C$ generated by morphisms that factor
through an object of $\widehat B$ outside $S$. Define $J_R$ in the
same way using $\widehat R$. Every factorization in $\widehat B$
is also a factorization in $\widehat R$, so
\[
J_B\subseteq J_R,
\qquad C/J_B\twoheadrightarrow C/J_R.
\]
The ideal $J_B$ acts as zero on every $\widehat B$-module
supported on $S$. Conversely, a $C/J_B$-module extends by zero
to $\widehat B$, since compositions through objects outside $S$
act as zero by the definition of $J_B$. Restriction and
extension by zero therefore give an equivalence
\[
\operatorname{mod}(C/J_B)\simeq
\{M\in\operatorname{mod}\widehat B:
\operatorname{supp}M\subseteq S\}.
\]
It preserves dimensions, isomorphism classes and indecomposability.
The same statement holds for $J_R$ and $\widehat R$.

For a one-parameter family with dimension vector supported on
$S$, each component outside $S$ has free rank zero and hence
vanishes. The action therefore factors through $C/J_B$.
Restriction and extension by zero commute with tensoring over
$\Lambda$, so the equivalence also applies to families and their
specializations.

For a fixed total dimension, there are only finitely many
dimension vectors on $S$. Since $\widehat B$ is tame,
the equivalence gives finitely many one-parameter families,
with finitely many exceptions, for the indecomposable
$C/J_B$-modules of that dimension. Thus $C/J_B$ is tame.
To pass to $C/J_R$, choose $\Lambda$-bases for each family.
If infinitely many fibres are annihilated by $J_R/J_B$, every
matrix entry describing its action vanishes at infinitely many
allowed parameter values. These entries lie in
$\Lambda=\kk[t,f(t)^{-1}]$, so they vanish identically and the
family is a $C/J_R$-family. Each remaining family contributes
only finitely many $C/J_R$-modules, which are added to the finite
exceptional set. Hence $C/J_R$ is tame.
Applying this argument to the support of each dimension vector
of $\widehat R$ proves that $\widehat R$ is tame.
\end{proof}

We use two comparison theorems. Let $R\to R/G$ be a Galois
covering of locally bounded categories, with $G$ torsion-free.
Then
\[
R\text{ is tame}\quad\Longleftrightarrow\quad R/G\text{ is tame}
\]
by \cite[Corollary~10.7\textnormal{(1)}]{Dro01}; see also
\cite[Theorem~9.1\textnormal{(1)}]{DO00}.

Geiss and Krause \cite[Theorem~5.3]{GK02} proved that
\[
\widehat A\text{ tame}\quad\Longrightarrow\quad A\text{ derived tame}.
\]
The converse holds when $\gldim A<\infty$.

\begin{lemma}\label{lem:repetitive-cover}
The universal covering $F:\U\to A$ of a finite-dimensional monomial
algebra induces a Galois covering
\[
\widehat F:\widehat\U\longrightarrow\widehat A,
\qquad \widehat\U/G\cong\widehat A,
\]
where $g(x,i)=(gx,i)$ for $g\in G$.
\end{lemma}

\begin{proof}
Nonzero paths form bases of the morphism spaces of $A$ and $\U$.
On each layer define $\widehat F$ by $F$. On adjacent layers, send
the functional dual to a lifted path $p$ to the functional dual
to $F(p)$.

Since the algebras are monomial, a product of two nonzero paths
is either another nonzero path or zero. Write $p^\vee$ for the
functional dual to a nonzero path $p$. With our convention on
path multiplication, the dual bimodule actions are
\[
q\cdot p^\vee=
\begin{cases}(p')^\vee,&p=p'q,\\0,&\text{otherwise},\end{cases}
\qquad
p^\vee\cdot r=
\begin{cases}(p'')^\vee,&p=rp'',\\0,&\text{otherwise}.
\end{cases}
\]
The remaining paths $p'$ and $p''$ may be trivial.
Each path has a unique lift once its starting vertex is fixed,
and its subpaths lift with it, so these products agree under $F$.
Products of two dual morphisms are zero in both repetitive
categories. Hence $\widehat F$ preserves composition.

For lifts $x,y$ of $a,b$, respectively, path lifting gives
\[
\bigoplus_{g\in G}\U(x,gy)\cong A(a,b),
\qquad
\bigoplus_{g\in G}\U(gx,y)\cong A(a,b).
\]
Only finitely many summands are nonzero, by local boundedness.
Dualizing these isomorphisms with $a$ and $b$ interchanged gives
the covering isomorphisms between adjacent layers.
For other pairs of layers both sides are zero.
The group acts freely on objects, and its orbits are indexed by
$(a,i)\in Q_0\times\Z$. Hence $\widehat F$ is a Galois
covering and $\widehat\U/G\cong\widehat A$.
\end{proof}

\begin{proof}[Proof of Theorem~\ref{thm:monomial}]
\textnormal{(i)$\Rightarrow$(ii).}
Suppose that \textnormal{(ii)} fails for some $B$.
By Theorem~\ref{thm:tree} and the tame--wild dichotomy, $B$ is
derived wild.
Derived wildness passes from a finite full subcategory of the
monomial universal cover to $A$
\cite[Lemma~2.8\textnormal{(a)}]{Fon22}.
Thus $A$ is derived wild, contrary to \textnormal{(i)}.

\textnormal{(ii)$\Rightarrow$(iii).}
Let $S$ be a finite set of vertices of $\U$.
By Lemma~\ref{lem:cover}, it lies in a finite connected convex tree
subcategory $B$ of finite global dimension.
The assumed Euler condition and Theorem~\ref{thm:tree}
show that $B$ is derived tame. Since $B$ has finite global dimension,
the Geiss--Krause theorem shows that $\widehat B$ is tame.
Lemma~\ref{lem:repetitive-local} gives tameness of $\widehat\U$.

\textnormal{(iii)$\Leftrightarrow$(iv).}
The deck group $G$ is free, hence torsion-free. The equivalence
follows from the covering theorem and Lemma~\ref{lem:repetitive-cover}.

\textnormal{(iv)$\Rightarrow$(i).}
This is the Geiss--Krause implication stated above.

\textnormal{(iv)$\Leftrightarrow$(v).}
The layer translation acts freely, and its group is isomorphic to
$\Z$. Apply the same covering theorem to
$\widehat A/\langle\nu\rangle\cong\mathrm T(A)$.
\end{proof}

\begin{remark}
When $|Q_1|-|Q_0|+1\geq1$, the covering quiver is an infinite
tree. Theorem~\ref{thm:monomial} then involves infinitely many
finite connected convex full subcategories. A single such
subcategory whose Euler form takes a negative value proves
derived wildness, as the next corollary shows. For one-cycle
algebras, Section~\ref{sec:fourier} gives a finite test for
the local Euler condition.
\end{remark}

\begin{corollary}\label{cor:monomial_der_wild}
Let $A=\kk Q/I$ be a finite-dimensional connected monomial
algebra. The following conditions are equivalent:
\begin{enumerate}[label=\textnormal{(\roman*)}]
\item $A$ is derived wild;
\item there is a finite connected convex full subcategory
$B\subseteq\U$ whose Euler form takes a negative value;
\item there is a finite connected convex full subcategory
$B\subseteq\U$ that is derived wild and whose proper connected
convex full subcategories are all derived tame.
\end{enumerate}
\end{corollary}

\begin{proof}
Theorem~\ref{thm:monomial} and the derived tame--wild dichotomy
give \textnormal{(i)$\Leftrightarrow$(ii)}. If \textnormal{(ii)}
holds, then Theorem~\ref{thm:tree} and the dichotomy give a
finite connected convex full subcategory that is derived wild.
Choose one with the least number of vertices. Each of its proper
connected convex full subcategories is convex in $\U$ and has
fewer vertices, so it is derived tame. This proves
\textnormal{(iii)}. Conversely, Theorem~\ref{thm:tree} shows
that the Euler form of the subcategory in \textnormal{(iii)}
takes a negative value, giving \textnormal{(ii)}.
\end{proof}

\begin{example}\label{ex:five-arm-star}
Let $A=\kk Q/\langle\alpha\beta\alpha\rangle$, where $Q$ is
the quiver
\begin{center}
\begin{tikzpicture}[quiver,scale=0.95]
  \begin{scope}[every node/.style={quiver vertex}]
    \node (1) at (-2,0) {$1$};
    \node (0) at (0,0) {$0$};
    \node (2) at (2.2, 1.3) {$2$};
    \node (3) at (2.5, 0.4) {$3$};
    \node (4) at (2.5, -0.4) {$4$};
    \node (5) at (2.2, -1.3) {$5$};
  \end{scope}

  \draw[->] (0) to[bend right=25] node[quiver label,above] {$\alpha$} (1);
  \draw[->] (1) to[bend right=25] node[quiver label,below] {$\beta$} (0);

  \draw[->] (0) -- node[quiver label,above left, inner sep=1pt] {$\gamma_1$} (2);
  \draw[->] (0) -- node[quiver label,above, inner sep=1pt] {$\gamma_2$} (3);
  \draw[->] (0) -- node[quiver label,below, inner sep=1pt] {$\gamma_3$} (4);
  \draw[->] (0) -- node[quiver label,below left, inner sep=1pt] {$\gamma_4$} (5);
\end{tikzpicture}
\end{center}
Every path of length five contains $\alpha\beta\alpha$ as a
subpath, so $A$ is finite-dimensional. Its underlying graph has
six vertices and six edges, hence $A$ is one-cycle.

Choose a lift $\tilde{0}$ of vertex $0$. The lift of $\beta$
ending at $\tilde{0}$ and the four lifts of $\gamma_i$ starting
there form the following subtree of the universal cover:
\begin{center}
\begin{tikzpicture}[quiver,scale=0.95]
  \begin{scope}[every node/.style={quiver vertex}]
    \node (1) at (-2,0) {$\tilde{1}$};
    \node (0) at (0,0) {$\tilde{0}$};
    \node (2) at (2.2, 1.3) {$\tilde{2}$};
    \node (3) at (2.5, 0.4) {$\tilde{3}$};
    \node (4) at (2.5, -0.4) {$\tilde{4}$};
    \node (5) at (2.2, -1.3) {$\tilde{5}$};
  \end{scope}

  \draw[->] (1) -- node[quiver label,above] {$\tilde{\beta}$} (0);
  \draw[->] (0) -- node[quiver label,above left, inner sep=1pt] {$\tilde{\gamma}_1$} (2);
  \draw[->] (0) -- node[quiver label,above, inner sep=1pt] {$\tilde{\gamma}_2$} (3);
  \draw[->] (0) -- node[quiver label,below, inner sep=1pt] {$\tilde{\gamma}_3$} (4);
  \draw[->] (0) -- node[quiver label,below left, inner sep=1pt] {$\tilde{\gamma}_4$} (5);
\end{tikzpicture}
\end{center}
It defines a finite connected convex full subcategory $B$.
No lifted relation is contained in this subtree, so $B$ is the
path algebra of the displayed five-armed star.
Let $x\in K_0(B)$ have coordinate $2$ at the centre and $1$
at each leaf. Since $B$ is hereditary,
\[
\chi_B(x)=2^2+5\cdot1^2-5\cdot2\cdot1=-1.
\]
Thus $A$ is derived wild by
Corollary~\ref{cor:monomial_der_wild}.
\end{example}

When the underlying graph has a cycle, the next proposition adds
a semichain condition to the equivalent conditions of
Theorem~\ref{thm:monomial}.

\begin{proposition}\label{prop:semichain}
Let $A=\kk Q/I$ be a finite-dimensional connected monomial algebra
with $|Q_1|-|Q_0|+1\geq1$.
Every finite connected convex full subcategory of $\U$ has
non-negative Euler form if and only if every such subcategory is
derived equivalent to a semichain algebra.
\end{proposition}

\begin{proof}
Assume the local Euler condition. The covering group is infinite,
since it is free of positive rank. Its free action on vertices
implies that the covering tree is infinite.

Suppose that a finite connected convex subcategory $C$ has one of
the exceptional or tubular derived types in Theorem~\ref{thm:tree}.
Choose a finite set $T\supseteq C_0$ of covering vertices with
$|T|>10$ and take the union of the undirected paths joining them.
By Lemma~\ref{lem:cover}, the resulting full subcategory $B$
is finite, connected and convex, contains $C$, and has $|B_0|>10$.
Since $C$ remains convex in $B$ and $\chi_B\geq0$, the extension property in
Theorem~\ref{thm:tree} gives $\rank K_0(B)\leq10$.
This contradicts $\rank K_0(B)=|B_0|>10$.
The classification in Theorem~\ref{thm:tree} now shows that every
finite connected convex subcategory is derived equivalent to a
semichain algebra.

Conversely, suppose every finite connected convex subcategory
is derived equivalent to a semichain algebra. Semichain algebras
are derived tame, and derived equivalence preserves derived
tameness. Theorem~\ref{thm:tree} therefore gives the local
Euler condition.
\end{proof}

\begin{remark}
The rank bound explains the conclusion of
Proposition~\ref{prop:semichain}. The exceptional hereditary
and tubular derived types in Theorem~\ref{thm:tree} have
Grothendieck group rank at most ten. Under the local Euler
condition, the same bound holds for every finite connected
convex full extension of a subcategory of one of these types.
In the infinite covering tree, such extensions have arbitrarily
many vertices. The finite-tree classification therefore leaves
only the semichain derived-equivalence classes.
\end{remark}

\section{A polynomial criterion for one-cycle monomial algebras}\label{sec:fourier}

Throughout this section, $A=\kk Q/I$ is a finite-dimensional
connected one-cycle monomial algebra, with covering group
$G=\langle\tau\rangle\cong\Z$. We use its nonzero paths to form
a matrix that tests the Euler forms of the finite convex
subcategories of $\U$.

Set $q=|Q_0|$ and choose a lift $x_i$ of each vertex $i\in Q_0$.
Every vertex of the cover has a unique expression $\tau^k x_i$.
For $d\in\Z$, define
\[
(D_d)_{ij}=\#\{\text{nonzero paths from }x_i\text{ to }\tau^d x_j\}.
\]
These counts include the trivial paths, so each $D_d$ has integer
entries. Only finitely many $D_d$ are nonzero, since every nonzero
path has length less than $L$ and the covering quiver is locally
finite. Put
\begin{equation}\label{eq:fourier-symbol}
P(z)=\sum_{d\in\Z}D_dz^d,
\qquad
H(z)=\frac12\bigl(P(z)+P(z^{-1})^T\bigr).
\end{equation}
For $|z|=1$, the matrix $H(z)$ is Hermitian, and its positive
semidefiniteness is tested over $\mathbb C$. Its coefficients
are determined by integer path counts, independently of the
characteristic of $\kk$.

The matrices can be computed directly from $Q$ and a finite set
$\mathcal R$ of monomial generators of $I$. A path is nonzero in
$A$ exactly when it contains no member of $\mathcal R$ as a
subpath. Starting from the trivial paths, we form paths of length
$\ell$ by appending arrows to the nonzero paths of length
$\ell-1$ and discarding those that contain a member of
$\mathcal R$. This process lists all nonzero paths by length
and terminates at the first empty list, since $A$ is finite-dimensional.
A one-cycle quiver has an arrow, and $I\subseteq R_Q^2$, so this
first occurs at some length $\ell\geq2$. At that point
$R_Q^\ell\subseteq I$, while the preceding list is nonempty.
Hence $\ell=L$, so the procedure also computes $L$.

\medskip
\noindent\textbf{Spanning-tree lifts.}
Choose a spanning tree $\mathcal T$ of the underlying graph and
a lift $\widetilde{\mathcal T}$ to the covering tree. From now on,
take its vertices as the lifts $x_i$. The unique arrow outside
$\mathcal T$, together with the unique undirected path in
$\mathcal T$ joining its
endpoints, gives a generator of the fundamental group.
Consequently, its lift starting in $\widetilde{\mathcal T}$ ends
in $\tau\widetilde{\mathcal T}$ for a generator $\tau$ of the deck
group. Give this arrow weight $1$ and every tree arrow weight $0$.
A path from $i$ to $j$ of total weight $d$ then lifts from $x_i$
to $\tau^d x_j$.
Thus $(D_d)_{ij}$ counts the nonzero paths from $i$ to $j$
of weight $d$. Every arrow has weight $0$ or $1$.
Since every nonzero path has length less than $L$, these
spanning-tree lifts give $D_d=0$ unless $0\leq d<L$.

\begin{theorem}\label{thm:fourier}
Let $A=\kk Q/I$ be a finite-dimensional connected one-cycle
monomial algebra, with $q=|Q_0|$ and
$L=\min\{\ell\geq2:R_Q^\ell\subseteq I\}$.
Form $H(z)$ as in \eqref{eq:fourier-symbol}, using the
spanning-tree lifts above. Set $M=L-1$ and define
\begin{equation}\label{eq:cayley-matrix}
z(t)=\frac{1+\mathrm i t}{1-\mathrm i t},
\qquad K(t)=(1+t^2)^M H(z(t)),
\end{equation}
and
\begin{equation}\label{eq:coefficient-test}
\det(sI_q+K(t))=s^q+c_1(t)s^{q-1}+\cdots+c_q(t).
\end{equation}
The following conditions are equivalent:
\begin{enumerate}[label=\textnormal{(\roman*)}]
\item every finite connected convex full subcategory of $\U$
has non-negative Euler form;
\item $H(z)\succeq0$ for every $|z|=1$;
\item $c_j(t)\geq0$ for every $t\in\mathbb R$ and $1\leq j\leq q$.
\end{enumerate}
Each $c_j$ lies in $\mathbb Q[t]$, and
$\deg c_j\leq2Mj=2j(L-1)$ when $c_j\ne0$.
These polynomials can be computed from the quiver and its monomial
relations, giving a finite test for the three conditions.
\end{theorem}

\begin{proof}
\textbf{1. Euler forms and path counts.}
Let $B$ be a finite connected convex subcategory of $\U$.
Since its quiver is acyclic, we can order its vertices as
$u_1,\ldots,u_m$ so that every arrow goes from a smaller index
to a larger one. Let $(P_B)_{ab}$ count the nonzero paths from
$u_a$ to $u_b$, and write $S_a$ for the simple at $u_a$.
For the right projective $P_a=e_{u_a}B$, the multiplicity of
$S_b$ is $(P_B)_{ab}$. Thus the Cartan matrix, whose columns are
the projective classes in the basis of simple classes, is
$C_B=P_B^T$. If $E_B$ is the Euler matrix in this basis, then
$\langle[P_a],[S_b]\rangle_B=\delta_{ab}$ gives
$C_B^TE_B=I$, hence $E_B=P_B^{-1}$.

The matrix $P_B$ is upper triangular with diagonal entries $1$, so $P_B^T$ is invertible. For any $y \in \mathbb{R}^m$, set $x = P_B^T y$. Then
\begin{equation}\label{eq:euler-congruence}
\chi_B(x) = \chi_B(P_B^T y) = (P_B^T y)^T P_B^{-1} P_B^T y = y^T P_B P_B^{-1} P_B^T y = y^T P_B^T y = y^T P_B y,
\end{equation}
where $y^T P_B^T y = y^T P_B y$ holds since $y^T P_B^T y$ is a scalar. Since $y \mapsto P_B^T y$ is an automorphism of $\mathbb{R}^m$, $\chi_B(x) \ge 0$ for all $x$ if and only if $y^T P_B y \ge 0$ for all $y \in \mathbb{R}^m$.

For the relation between the Euler and Cartan forms, see
\cite[\S1.6, proof of Lemma~1]{BBD00}.

\noindent\textbf{2. Vectors of finite support.}
Let $y=(y_k)_{k\in\Z}$ be a finitely supported family with
$y_k\in\mathbb R^q$, and set
\[
\Phi(y)=\sum_{k,l\in\Z}y_k^TD_{l-k}y_l.
\]
Assign $(y_k)_i$ to the vertex $\tau^k x_i$.
Applying $\tau^{-k}$ gives a bijection from paths between
$\tau^k x_i$ and $\tau^l x_j$ to paths between $x_i$ and
$\tau^{l-k}x_j$. By Lemma~\ref{lem:cover}, the support of $y$
lies in a finite connected convex subtree, which contains every
path between its vertices. Its path-count form on $y$ is therefore
$\Phi(y)$.

Conversely, every vector on a finite convex subcategory gives
such a family. By \ref{eq:euler-congruence}, the local Euler
condition is equivalent to $\Phi(y)\geq0$ for all finitely
supported real families.

\noindent\textbf{3. The unit-circle condition.} Define $H_d = \frac{1}{2}(D_d + D_{-d}^T)$ so that $H(z) = \sum_{d \in \mathbb{Z}} H_d z^d$. For any complex sequence $y = u + iv$ with finite support, $H_{-d} = H_d^T$ yields $\Psi(y) := \sum_{k,l} y_k^* H_{l-k} y_l = \Phi(u) + \Phi(v)$. Thus, non-negativity over real sequences is equivalent to non-negativity over complex sequences. 

For $f_y(z) = \sum_k y_k z^{-k}$, orthogonality of characters on $S^1$ gives
\begin{equation}\label{eq:parseval_psi}
\Psi(y) = \frac{1}{2\pi} \int_0^{2\pi} f_y(e^{i\theta})^* H(e^{i\theta}) f_y(e^{i\theta}) \, d\theta.
\end{equation}
If $H(z) \ge 0$ for all $z \in S^1$, then $\Psi(y) \ge 0$.

Conversely, assume $\Psi(y) \ge 0$ for all finitely supported sequences. Given any $z_0 \in S^1$ and $v \in \mathbb{C}^q$, define $y_k = z_0^k v$ for $0 \le k < N$ and $y_k = 0$ otherwise. Counting pairs with $l - k = d$ yields:
\[
0 \le \frac{\Psi(y)}{N} = \sum_{|d| < N} \left(1 - \frac{|d|}{N}\right) z_0^d \, v^* H_d v.
\]
Since only finitely many $H_d$ are non-zero, taking $N \to \infty$ gives $v^* H(z_0) v \ge 0$. As $z_0 \in S^1$ and $v \in \mathbb{C}^q$ were arbitrary, $H(z) \ge 0$ on the entire unit circle $S^1$.

\noindent\textbf{4. A finite polynomial test.}
For the spanning-tree lifts, $H_d=0$ when $|d|>L-1$.
With $M=L-1$ and $K(t)$ as in \eqref{eq:cayley-matrix}, the terms
with exponents $d$ and $-d$, for $0\leq d\leq M$, become
\[
(1+t^2)^M z(t)^{\pm d}
=(1+t^2)^{M-d}(1\pm\mathrm i t)^{2d}.
\]
It follows that the entries of $K(t)$ lie in $\mathbb Q(\mathrm i)[t]$
and have degree at most $2M$. For all real $t$, the factor
$(1+t^2)^M$ is positive and $K(t)$ is Hermitian.

In \eqref{eq:coefficient-test}, the coefficient $c_j(t)$ is the
sum of the principal minors
of order $j$ of $K(t)$; equivalently, it is the $j$th elementary
symmetric function of the eigenvalues of $K(t)$. Each entry has
degree at most $2M$,
so $c_j(t)\in\mathbb Q(\mathrm i)[t]$ and
$\deg c_j\leq2Mj$ when $c_j$ is nonzero.
Since $K(t)$ is Hermitian for all real $t$, its eigenvalues and each
$c_j(t)$ are real on $\mathbb R$. The imaginary part of $c_j$
is therefore a polynomial vanishing on $\mathbb R$, so
$c_j\in\mathbb Q[t]$.

For a fixed real $t$, let $\lambda_1,\ldots,\lambda_q$ be the
eigenvalues of $K(t)$. If they are all non-negative, then the
coefficients of $\prod_{a=1}^q(s+\lambda_a)$ are non-negative.
Conversely, if all $c_j(t)$ are non-negative, the right side of
\eqref{eq:coefficient-test} is strictly positive for $s>0$.
A negative eigenvalue $\lambda_a$ would give the positive root
$s=-\lambda_a$, a contradiction. Thus
\[
K(t)\succeq0
\quad\Longleftrightarrow\quad
c_j(t)\geq0\quad(1\leq j\leq q).
\]

As $t$ runs through $\mathbb R$, $z(t)$ covers the unit circle
except for $-1$, and $z(t)\to-1$ as $t\to+\infty$.
If $H(z(t))\succeq0$ for all real $t$, then for every
$v\in\mathbb C^q$, continuity gives
\[
v^*H(-1)v=\lim_{t\to+\infty}v^*H(z(t))v\geq0.
\]
Consequently,
\[
H(z)\succeq0\text{ for all }|z|=1
\quad\Longleftrightarrow\quad
c_j(t)\geq0\text{ for all }t\in\mathbb R\text{ and }1\leq j\leq q.
\]

A nonzero real polynomial is non-negative on $\mathbb R$
exactly when its leading coefficient is positive and every real
root has even multiplicity. Indeed, an odd-multiplicity real root
changes the sign, while an even-multiplicity root does not.
The zero polynomial is non-negative as well. For a polynomial
in $\mathbb Q[t]$, these conditions can be checked by square-free
factorization and exact real-root isolation. The criterion thus
reduces to $q$ finite polynomial tests with the stated degree bounds.
\end{proof}

\begin{remark}
    The choice of lifts does not affect the criterion. If
$x'_i=\tau^{s_i}x_i$ and
$D(z)=\operatorname{diag}(z^{s_1},\ldots,z^{s_q})$, then
\[
P'(z)=D(z)P(z)D(z)^{-1},\qquad
H'(z)=D(z)H(z)D(z)^*\quad(|z|=1).
\]
Since $D(z)$ is unitary on the unit circle, $H'$ and $H$ are
positive semidefinite at the same points. With $M=L-1$ fixed,
similarity also gives
\[
\det\!\left(sI_q+(1+t^2)^M H'(z(t))\right)
=\det(sI_q+K(t)).
\]
This is an identity of rational functions in $t$, so the
determinant coefficients are the same polynomials $c_j(t)$.
The spanning-tree lifts ensure that the entries of $K(t)$
themselves are polynomials of degree at most $2M$.
Replacing $\tau$ by $\tau^{-1}$ replaces $z$ by $z^{-1}$ and
leaves the unit-circle condition unchanged.
\end{remark}

Theorems~\ref{thm:monomial} and~\ref{thm:fourier} together prove
Theorem~\ref{thm:main}.

\begin{corollary}\label{cor:nonoriented-cycle}
Let $A=\kk Q/I$ be a finite-dimensional connected one-cycle
monomial algebra whose unique cycle is not oriented. Using
spanning-tree lifts, define
\[
K_1(t)=(1+t^2)H\!\left(\frac{1+\mathrm i t}{1-\mathrm i t}\right),
\qquad
\det(sI_q+K_1(t))=s^q+\sum_{j=1}^q b_j(t)s^{q-j}.
\]
Then $b_j\in\mathbb Q[t]$ and $\deg b_j\leq2j$ whenever
$b_j\ne0$. Moreover, $A$ is derived tame if and only if
$b_j(t)\geq0$ for all $t\in\mathbb R$ and $1\leq j\leq q$.
\end{corollary}

\begin{proof}
The quiver $Q$ has no oriented cycles, so a directed path visits
each vertex at most once. It therefore uses the unique arrow
outside the spanning tree at most once. Hence $D_d=0$ for
$d\notin\{0,1\}$ and $P(z)=D_0+zD_1$. It follows that
\[
K_1(t)=\frac12\bigl((1+t^2)(D_0+D_0^T)
 +(1+\mathrm i t)^2D_1+(1-\mathrm i t)^2D_1^T\bigr).
\]
Its entries lie in $\mathbb Q(\mathrm i)[t]$ and have degree
at most two. The coefficient $b_j$ is the sum of its principal
minors of order $j$, so it has degree at most $2j$ when nonzero.
Since $K_1(t)$ is Hermitian for real $t$, these coefficients
are real on $\mathbb R$ and therefore belong to $\mathbb Q[t]$.

For the matrix $K(t)$ and coefficients $c_j(t)$ in
Theorem~\ref{thm:fourier}, we have
\[
K(t)=(1+t^2)^{L-2}K_1(t),\qquad
c_j(t)=(1+t^2)^{j(L-2)}b_j(t).
\]
The factors are positive for real $t$, so the two sets of
polynomial inequalities are equivalent. Theorem~\ref{thm:main}
now gives the result.
\end{proof}

\section{Applications to one-cycle monomial algebras}\label{sec:examples}

We first apply the criterion to truncated cycles and to path
algebras of acyclic orientations of a cycle. We then consider
a string algebra with a minimal cubic relation, give a second
proof of its derived tameness by flat deformation, and
construct its finite cyclic covers.

\begin{proposition}\cite[Lemma~3.2]{BGV21}
\label{prop:truncated-cycles}
Let $Q$ be an oriented cycle on $n\geq1$ vertices and let
$A_{n,r}=\kk Q/R_Q^r$, where $r\geq2$.
Then $A_{n,r}$ is derived tame if and only if $r=2$.
\end{proposition}

\begin{proof}
The universal cover is the oriented infinite line with all paths
of length $r$ zero, independently of $n$. By
Theorem~\ref{thm:monomial}, it suffices to take $n=1$. In that case
the matrix has one entry:
\[
H_r(e^{i\theta})=\sum_{j=0}^{r-1}\cos(j\theta).
\]
For $r=2$ this equals $1+\cos\theta\geq0$. For $r>2$, choose
$\pi/(r-1)<\theta<2\pi/r$; then
\[
H_r(e^{i\theta})=
\frac{\sin(r\theta/2)\cos((r-1)\theta/2)}{\sin(\theta/2)}<0.
\]
The result follows from Theorem~\ref{thm:main}.
\end{proof}

For $r=3$, the polynomial test in Theorem~\ref{thm:fourier}
with $M=2$ gives
\[
(1+t^2)^2H_3\left(\frac{1+it}{1-it}\right)
=t^4-4t^2+3=(t^2-1)(t^2-3).
\]
This polynomial is negative for $1<t^2<3$, although
$H_3(1)=3>0$. Thus a test at $z=1$ alone does not decide
positive semidefiniteness on the unit circle.

We next consider a cycle whose arrows do not form an oriented
cycle, with no relations.

\begin{proposition}\cite[\S1.4(a)]{NdP01}\label{prop:nonoriented-cycle-path}
Let $Q$ be an acyclic orientation of a cycle.
Then $\kk Q$ is derived tame.
\end{proposition}

\begin{proof}
Put $A=\kk Q$ and $q=|Q_0|$. Choose a spanning tree of the underlying graph.
Give the unique arrow outside this tree weight $1$, and give
every tree arrow weight $0$, as in Section~\ref{sec:fourier}.
Define
\[
N(z)_{ij}=\sum_{a:i\to j}z^{w(a)}.
\]
The entries of $N(z)^r$ count paths of length $r$, with each path
contributing $z$ raised to its total weight.

Since $Q$ has no oriented cycles, every path has length less
than $q$. Thus $A$ is finite-dimensional and $N(z)^q=0$.
There are no relations, so the path-count matrix is
\[
P(z)=I_q+N(z)+\cdots+N(z)^{q-1}
    =(I_q-N(z))^{-1}.
\]

Fix $z$ with $|z|=1$, and write $R(z)=I_q-N(z)$.
Then $R(z)$ is invertible and
$$H(z)=\frac12(P(z)+P(z)^*), $$
where $*$ denotes conjugate
transpose. Hence
\begin{align*}
2R(z)^*H(z)R(z)
 &=R(z)^*P(z)R(z)+R(z)^*P(z)^*R(z)\\
 &=R(z)^*+R(z)\\
 &=2I_q-N(z)-N(z)^*.
\end{align*}
Every vertex of the underlying graph has degree two.
Consequently, for $x\in\mathbb C^q$,
\[
x^*\bigl(2I_q-N(z)-N(z)^*\bigr)x
 =\sum_{a:i\to j}
   \left|x_i-z^{w(a)}x_j\right|^2.
\]
This proves that $2I_q-N(z)-N(z)^*$ is positive semidefinite.
Since $R(z)$ is invertible, $H(z)$ is positive semidefinite
for every $|z|=1$. It shows that A is derived tame by Theorem~\ref{thm:main}.
\end{proof}

\begin{remark}
In the setting of Proposition~\ref{prop:nonoriented-cycle-path},
the matrix $H(z)$ in \eqref{eq:fourier-symbol} is positive
definite for $|z|=1$ and $z\ne1$, while $H(1)$ has corank one.

Indeed, suppose that the sum of squares in the proof vanishes.
Every tree arrow has weight zero, so its endpoints have equal
coordinates. Since the spanning tree is connected, all coordinates
of $x$ have a common value $c$. The remaining arrow gives $c=zc$.
If $z\ne1$, then $c=0$, so $2I_q-N(z)-N(z)^*$ is positive definite.
If $z=1$, its kernel consists of the constant vectors and has
dimension one. The invertible congruence used in the proof
preserves positive definiteness and the dimension of the kernel,
so both conclusions hold for $H(z)$.
\end{remark}

In contrast to Proposition \ref{prop:truncated-cycles}, where cubic relations on an oriented cycle yield derived wildness, a cubic relation located along a tree branch outside the central cycle does not destroy derived tameness. The following example illustrates this phenomenon.

\begin{example}\label{ex:nonquadratic}
Let $Q$ be the quiver
\begin{center}
\begin{tikzpicture}[quiver,x=1cm,y=1cm]
\node[quiver vertex] (v9) at (0,0) {$9$};
\node[quiver vertex] (v8) at (1.35,0) {$8$};
\node[quiver vertex] (v7) at (2.70,0) {$7$};
\node[quiver vertex] (v1) at (4.05,0) {$1$};
\node[quiver vertex] (v2) at (5.40,0) {$2$};
\node[quiver vertex] (v3) at (6.75,0) {$3$};
\node[quiver vertex] (v4) at (8.10,0) {$4$};
\node[quiver vertex] (v5) at (9.45,0) {$5$};
\node[quiver vertex] (v6) at (10.80,0) {$6$};
\node[quiver vertex] (v10) at (6.75,1.30) {$10$};
\node[quiver vertex] (v11) at (8.10,1.30) {$11$};
\node[quiver vertex] (v12) at (9.45,1.30) {$12$};
\node[quiver vertex] (v13) at (10.80,1.30) {$13$};
\draw[->] (v1) -- node[quiver label,above] {$\beta_1$} (v7);
\draw[->] (v7) -- node[quiver label,above] {$\beta_2$} (v8);
\draw[->] (v8) -- node[quiver label,above] {$\beta_3$} (v9);
\draw[->] (v1) -- node[quiver label,above] {$\alpha_1$} (v2);
\draw[->] (v2) -- node[quiver label,above] {$\alpha_2$} (v3);
\draw[->] (v3) -- node[quiver label,above] {$\alpha_3$} (v4);
\draw[->] (v4) -- node[quiver label,above] {$\alpha_4$} (v5);
\draw[->] (v5) -- node[quiver label,above] {$\alpha_5$} (v6);
\draw[->,rounded corners=5pt] (v6.south) -- (10.80,-0.65)
  -- node[quiver label,below] {$\alpha_6$} (4.05,-0.65) -- (v1.south);
\draw[->] (v3) -- node[quiver label,left] {$\gamma_1$} (v10);
\draw[->] (v10) -- node[quiver label,above] {$\gamma_2$} (v11);
\draw[->] (v13) -- node[quiver label,above] {$\delta_2$} (v12);
\draw[->] (v12) -- node[quiver label,right] {$\delta_1$} (v5);
\end{tikzpicture}
\end{center}
With $\alpha$-subscripts read modulo $6$, put
\[
A=\kk Q/\langle\alpha_i\alpha_{i+1}\ (1\leq i\leq6),
\ \beta_1\beta_2\beta_3\rangle.
\]
Then $A$ is a derived-tame one-cycle string algebra with a
minimal cubic relation.
\end{example}

\begin{proof}
The graph has $13$ vertices and $13$ edges, with at most two
arrows entering and at most two arrows leaving each vertex. Its nonzero
length-two paths are
\[
\mathcal L_2\cup\{\beta_2\beta_3\},\qquad
\mathcal L_2=
\{\alpha_6\beta_1,\ \beta_1\beta_2,\
\alpha_2\gamma_1,\ \gamma_1\gamma_2,\
\delta_2\delta_1,\ \delta_1\alpha_5\}.
\]
In this list, each arrow has at most one nonzero successor
and at most one nonzero predecessor, so $A$ is a string algebra.
Every path of length four is zero.
Both length-two subpaths of $\beta_1\beta_2\beta_3$ are nonzero,
so this is a minimal cubic relation. Thus the ideal is
admissible and the algebra is not quadratic.

\emph{Proof by the matrix criterion.}
We compute the matrix in Theorem~\ref{thm:fourier}. Give $\alpha_6$ weight $1$ and all
other arrows weight $0$. In the vertex order
$(1,\ldots,6\mid7,\ldots,13)$, write
\[
2H(z)=\begin{pmatrix}W(z)&U(z)\\U(z)^*&V\end{pmatrix},
\qquad |z|=1.
\]
The only nonzero paths between cycle vertices are the trivial
paths and the arrows $\alpha_i$. Thus $W$ has diagonal $2$, entries
$W_{i,i+1}=W_{i+1,i}=1$ for $1\leq i\leq5$,
$W_{6,1}=z$, $W_{1,6}=z^{-1}$, and all other entries zero.
Write $J_m$ for the all-ones matrix and $u_i$ for the standard
basis vectors of $\mathbb C^6$. The remaining blocks are
\[
V=\operatorname{diag}(I_3+J_3,I_2+J_2,I_2+J_2),
\qquad b=(1,1,0),\quad c=(1,1),
\]
\[
U(z)=\bigl((u_1+zu_6)b\mid(u_2+u_3)c\mid(u_5+u_6)c\bigr).
\]
The blocks of $V$ come from the paths within the three branches.
For example, the branch on $7,8,9$ has the paths $\beta_2$,
$\beta_3$ and $\beta_2\beta_3$, so its block is $I_3+J_3$.
The first block of $U$ counts the paths from $1$ and $6$ to
$7$ and $8$; the relation $\beta_1\beta_2\beta_3=0$ excludes
paths from either vertex to $9$. The second block counts the
paths from $2$ and $3$ to $10$ and $11$. The third block comes
from paths in the opposite direction, from $12$ and $13$ to
$5$ and $6$, through the term $P(z^{-1})^T$ in $2H(z)$.
All paths contributing to the second and third blocks have weight zero.

The matrix $I_m+J_m$ has eigenvalue $m+1$ once and eigenvalue
$1$ with multiplicity $m-1$, and
\[
(I_m+J_m)^{-1}=I_m-\frac{1}{m+1}J_m.
\]
Thus $V$ is positive definite. The Schur complement of $V$ is
$S=W-UV^{-1}U^*$.
For $x\in\mathbb C^6$ and $y\in\mathbb C^7$, we have
\[
\begin{pmatrix}x\\y\end{pmatrix}^{\!*}
\begin{pmatrix}W&U\\U^*&V\end{pmatrix}
\begin{pmatrix}x\\y\end{pmatrix}
=x^*Sx+(y+V^{-1}U^*x)^*V(y+V^{-1}U^*x).
\]
Thus it remains to check that $S$ is positive semidefinite.
The identities
\[
b(I_3+J_3)^{-1}b^T=1,\qquad
c(I_2+J_2)^{-1}c^T=\frac23
\]
give the following contributions to $UV^{-1}U^*$:
$\left(\begin{smallmatrix}1&z^{-1}\\z&1\end{smallmatrix}\right)$
on indices $(1,6)$, and $\frac23J_2$ on each of $(2,3)$ and
$(5,6)$. After subtraction from $W$, all entries containing $z$
cancel. Thus
\begin{align*}
x^*Sx={}&|x_1+x_2|^2+\tfrac13|x_2+x_3|^2
+|x_3+x_4|^2
+|x_4+x_5|^2+\tfrac13|x_5+x_6|^2\geq0
,~ x\in\mathbb C^6.
\end{align*}
Thus $H(z)\succeq0$ on the entire unit circle, and $A$ is
derived tame by Theorem~\ref{thm:main}.

\emph{Proof by a flat deformation.}
We adapt the local construction in \cite[Lemma~2.4]{NdP01}
and verify flatness directly. Add an arrow $\theta:7\to9$
to obtain $Q^+$, and set
\begin{equation}\label{eq:family}
\mathcal A=
\kk[t]Q^+/\langle\alpha_i\alpha_{i+1},\
\beta_2\beta_3-t\theta,\ \beta_1\theta\rangle.
\end{equation}
A $\kk[t]$-basis is
\[
\{e_1,\ldots,e_{13}\}\ \cup\ Q^+_1\ \cup\ \mathcal L_2
\ \cup\
\{\alpha_6\beta_1\beta_2,\
\alpha_2\gamma_1\gamma_2,\
\delta_2\delta_1\alpha_5\}.
\]

To verify this basis, use the rules
$\alpha_i\alpha_{i+1}\mapsto0$,
$\beta_2\beta_3\mapsto t\theta$ and $\beta_1\theta\mapsto0$.
The only arrow entering vertex $7$ is $\beta_1$, and no arrow
leaves vertex $9$. Thus every path of length at least two that
contains $\theta$ has $\beta_1\theta$ as a subpath and reduces
to zero. Every reduction either gives zero or shortens the path,
so the reduction process terminates.

The only overlapping left sides occur in
$\alpha_i\alpha_{i+1}\alpha_{i+2}$, and reducing either pair
gives zero. Reductions on disjoint subpaths commute.
Also, $\beta_1\beta_2\beta_3$ reduces first to $t\beta_1\theta$
and then to zero. Hence the displayed paths are exactly the
irreducible paths, and every element has a unique normal form
as a $\kk[t]$-linear combination of them. Since there are
$13+14+6+3=36$ such paths, $\mathcal A$ is free of rank $36$
over $\kk[t]$ and the family is flat.

For $\lambda\ne0$, eliminating
$\theta=\lambda^{-1}\beta_2\beta_3$ identifies
$\mathcal A/(t-\lambda)\mathcal A$ with $A$. The zero fibre is
\[
A_0=\kk Q^+/\langle\alpha_i\alpha_{i+1},
\ \beta_2\beta_3,\ \beta_1\theta\rangle.
\]
The zero fibre $A_0$ is gentle: at most two arrows start and at
most two end at each vertex, and its nonzero paths of length two
are precisely those in $\mathcal L_2$. This list and the displayed
relations give each arrow at most one nonzero and one zero
successor, with the same bounds for predecessors.
Since $A_0$ is derived tame \cite{BM03} and every nonzero fibre
is isomorphic to $A$, \cite[Corollary~1.5]{BD03} implies that
$A$ is derived tame.

\end{proof}

\begin{remark}\label{rem:schur-reduction}
Let $A$ be a one-cycle monomial algebra. Write $C_0$ for the
vertices on its unique cycle and put $F_0=Q_0\setminus C_0$.
Assume $F_0\ne\varnothing$, and choose the arrow of weight $1$
on the cycle. Suppose that no nonzero path with both endpoints
in $F_0$ passes through this arrow. With the vertices ordered
as $(C_0,F_0)$, write
\[
2H(z)=\begin{pmatrix}W(z)&U(z)\\ U(z)^*&V\end{pmatrix}.
\]
Every path contributing to $V$ has weight zero, so $V$ is a
constant real symmetric matrix. If $V$ is positive definite,
the completion of the square used in
Example~\ref{ex:nonquadratic} gives
\[
H(z)\succeq0\quad\Longleftrightarrow\quad
W(z)-U(z)V^{-1}U(z)^*\succeq0\qquad(|z|=1).
\]
This reduces the test to a matrix of order $|C_0|$.

To check positive definiteness of $V$, suppose further that
every nonzero path with both endpoints in $F_0$
stays in $F_0$, and let $B$ be the algebra of the induced bound
subquiver on $F_0$. Then $B$ is a product of finite tree algebras
and $V=P_B+P_B^T$. By \eqref{eq:euler-congruence}, applied
to each component, $V$ is positive definite if and only if the
Euler form of $B$ is positive definite. In particular, this holds
if each component is a directed path with no relations:
then $V$ is a direct sum of matrices $I_m+J_m$.
These conditions hold in Example~\ref{ex:nonquadratic}.
\end{remark}

Finite cyclic covers preserve derived tameness. Applying the following
proposition to Example~\ref{ex:nonquadratic} gives an infinite family.

\begin{proposition}\label{prop:finite-cyclic-covers}
Let $A$ be a finite-dimensional connected one-cycle monomial
algebra, with universal cover $\U$ and deck group
$\langle\tau\rangle$. For $n\geq1$, let
$A^{(n)}=\U/\langle\tau^n\rangle.$
Then $A^{(n)}$ is a finite-dimensional connected one-cycle
monomial algebra, and
\[
A^{(n)}\text{ is derived tame}
\quad\Longleftrightarrow\quad
A\text{ is derived tame}.
\]
\end{proposition}

\begin{proof}
The quotient quiver is connected and has $n|Q_0|$ vertices and
$n|Q_1|$ arrows. Its relations are the lifts of the monomial
relations of $A$, so every nonzero path has length less than $L$
and $A^{(n)}$ is finite-dimensional and monomial. The identity
$n|Q_1|-n|Q_0|+1=1$ shows that it is one-cycle.

The projection $\U\to A^{(n)}$ is a bound-quiver covering
with deck group $\langle\tau^n\rangle$. Its source has a tree
as its underlying graph and is therefore simply connected.
Hence it is the universal cover of $A^{(n)}$. Thus $A$ and
$A^{(n)}$ have the same local Euler condition, and
Theorem~\ref{thm:monomial} gives the equivalence.
\end{proof}

\begin{corollary}\label{cor:nonquadratic-family}
For each $n\geq1$, there is a derived-tame one-cycle string
algebra with $13n$ vertices, dimension $36n$, and $n$ minimal
cubic relations. No two algebras in this family are derived equivalent.
\end{corollary}

\begin{proof}
Let $A$ be the algebra in Example~\ref{ex:nonquadratic}, and
take $A^{(n)}$ as in Proposition~\ref{prop:finite-cyclic-covers}.
An explicit presentation has vertices
\[
(j,i),\qquad j\in\mathbb Z/n\mathbb Z,\quad 1\leq i\leq13.
\]
For each $j$, copy all arrows of the example except
$\alpha_6$, and replace that arrow by
\[
\alpha_{j,6}:(j,6)\longrightarrow(j+1,1).
\]
Write $\alpha_{j,i}$, $\beta_{j,i}$, $\gamma_{j,i}$ and
$\delta_{j,i}$ for the resulting arrows. The relations are
\[
\begin{gathered}
\alpha_{j,i}\alpha_{j,i+1}=0
\quad(1\leq i\leq5),\qquad
\alpha_{j,6}\alpha_{j+1,1}=0,\qquad
\beta_{j,1}\beta_{j,2}\beta_{j,3}=0
\quad(j\in\mathbb Z/n\mathbb Z).
\end{gathered}
\]
The covering preserves the number of arrows at each vertex
and the allowed continuations of each arrow, so $A^{(n)}$ is a
string algebra. Both length-two subpaths of each cubic relation
are nonzero, giving $n$ distinct minimal cubic relations.
Every path of length four is zero.

Each nonzero path of $A$, including a trivial path, has exactly
$n$ lifts to this quiver. Thus
$\dim_\kk A^{(n)}=n\dim_\kk A=36n$.
Derived tameness follows from
Proposition~\ref{prop:finite-cyclic-covers}.
Since derived equivalence preserves the rank of the Grothendieck
group, the equality $\rank K_0(A^{(n)})=13n$ distinguishes the
derived equivalence classes for different values of $n$.
\end{proof}

\begin{remark}\label{rem:1-cycle-extension} 
The construction in Corollary~\ref{cor:nonquadratic-family} is entirely general. If $A$ is any derived tame one-cycle monomial algebra with $d$ vertices, the sequence of finite cyclic coverings $A^{(n)} = \tilde{A}/\langle\tau^n\rangle$ yields an infinite family of pairwise non-derived equivalent derived tame algebras, distinguished by $\operatorname{rank} K_0(A^{(n)}) = nd$. Consequently, every derived tame one-cycle monomial algebra gives rise to an infinite family of this type.
\end{remark}

\end{document}